\documentclass[a4paper,11pt,twoside]{amsart}

\usepackage{amsmath}
\usepackage{amsthm}
\usepackage{amssymb}
\usepackage[all]{xy}
\usepackage{hyperref}
\usepackage{rotating}

\newtheorem{theorem}{Theorem}[section]
\newtheorem{proposition}[theorem]{Proposition}

\newtheorem{corollary}[theorem]{Corollary}
\newtheorem{conjecture}[theorem]{Conjecture}
\newtheorem*{orlconj}{Conjecture \ref{orlov-conj}}

\numberwithin{equation}{section}

\theoremstyle{definition}
\newtheorem{definition}[theorem]{Definition}
\newtheorem{remark}[theorem]{Remark}

\newtheorem{question}[theorem]{Question}

\newcommand{\Db}{{\rm D}^{\rm b}}

\newcommand{\Br}{{\rm Br}}

\newcommand{\CH}{{\rm CH}}
\newcommand{\Pic}{{\rm Pic}}

\newcommand{\coh}{{\cat{Coh}}}

\newcommand{\corr}{{\rm Corr}}

\newcommand{\Hom}{{\rm Hom}}
\newcommand{\Spec}{{\rm Spec}}

\newcommand{\Td}{{\rm Td}}
\renewcommand{\dim}{{\rm dim}\,}

\newcommand{\id}{{\rm id}}

\newcommand{\cat}[1]{\begin{bf}#1\end{bf}}

\newcommand{\cal}{\mathcal}
\newcommand{\ka}{{\cal A}}
\newcommand{\kb}{{\cal B}}

\newcommand{\ke}{{\cal E}}
\newcommand{\kf}{{\cal F}}

\newcommand{\km}{{\cal M}}
\newcommand{\ko}{{\cal O}}

\newcommand{\ZZ}{\mathbb{Z}}
\newcommand{\QQ}{\mathbb{Q}}

\newcommand{\CC}{\mathbb{C}}

\newcommand{\FF}{\mathbb{F}}
\newcommand{\LL}{\mathbb{L}}
\newcommand{\PP}{\mathbb{P}}

\newcommand{\Alb}{{\rm Alb}}

\begin{document}

\title{Categorical representability and intermediate Jacobians of Fano threefolds}

\author[M.\ Bernardara, M.\ Bolognesi]{Marcello Bernardara and Michele Bolognesi}

\address{M.Be.: Univerist\"at Duisburg--Essen, Fakult\"at f\"ur Mathematik. Universit\"atstr. 2, 45117 Essen (Germany)}
\email{marcello.bernardara@uni-due.de}

\address{M.Bo.: IRMAR, Universit\'e de Rennes 1.
263 Av. G\'en\'eral Leclerc, 35042 Rennes CEDEX (France)} 
\email{michele.bolognesi@univ-rennes1.fr}

\begin{abstract}
We define, basing upon semiorthogonal decompositions of $\Db(X)$, categorical representability
of a projective variety $X$ and describe its relation with classical representabilities of the Chow ring.
For complex threefolds satisfying both classical and categorical representability assumptions, we
reconstruct the intermediate Jacobian from the semiorthogonal decomposition. We discuss finally how
categorical representability can give useful information on the birational properties of $X$ by
providing examples and stating open questions.
\end{abstract}

\maketitle

\section{Introduction}

These notes arise from the attempt to extend the results of \cite{marcello-bolo-conicbd} to a wider class of complex threefolds with
negative Kodaira dimension. If $Y \to S$ is a conic bundle and $S$ is rational, a semiorthogonal decomposition of $\Db(Y)$
by derived categories of curves and exceptional objects gives a splitting of the intermediate Jacobian as the direct
sum of the Jacobians of the curves \cite[Thm 1.1]{marcello-bolo-conicbd}. This result is based on the relation between fully
faithful functors $\Db(\Gamma) \to \Db(Y)$ (where $\Gamma$ is a smooth projective curve) and algebraic cycles on $Y$.
First of all, using the Chow--K\"unneth decompositions of the motives of $\Gamma$ and $Y$, we get from such functor an isogeny
between $J(\Gamma)$ and an abelian subvariety of $J(Y)$. Secondly, using the universal isomorphism between $A^2_{\ZZ}(Y)$ and a Prym variety
and the incidence property, we prove that such an isogeny is indeed an injective morphism of principally polarized
abelian varieties. Finally, the existence of the mentioned semiorthogonal decomposition assures the splitting of the intermediate Jacobian.
It turns out that the properties needed to prove this result are enjoyed also by threefolds other than conic bundles. One of the aims of this paper is to describe certain varieties satisfying those representability assumptions.
 
In a generalization attempt, we define a new notion of representability based on semiorthogonal decompositions, which we expect to
carry useful geometrical insights also in higher dimensions.
Let $X$ be a smooth projective variety of dimension $n$. We define \it categorical representability in
(co)dimension $m$ \rm for $X$, roughly by requiring that the derived category $\Db(X)$ admits a semiorthogonal decomposition
by categories appearing in smooth projective varieties of dimension $m$ (resp. $n-m$). 

The easiest case is of course representability in dimension 0. This is equivalent to say that $X$ admits
a full exceptional sequence of a finite number, say $l$, of objects. In this case we have $K_0(X) = \ZZ^l$. This
is indeed a very strong notion and gives rise to intriguing questions to explore even for surfaces.

Various notions of representability of the group $A^i_{\ZZ}(X)$ of algebraically trivial cycles of codimension
$i$ on $X$ have appeared throughout the years in the literature, and it seems interesting to understand
their interactions with categorical representability, as our examples suggest.
Roughly speaking, \it weak representability \rm for $A^i_{\ZZ}(X)$ is given by an algebraic map
$J(\Gamma) \to A^i_{\ZZ}(X)$ whose kernel is an algebraic group, for an algebraic curve $\Gamma$. Working
with rational coefficients (that is, with $A^i_{\QQ}$) gives the notion of \it rational representability.
Algebraic representability \rm
requires the existence of a universal regular isomorphism $A^i_{\ZZ}(X) \to A$ onto an abelian variety $A$.
Finally, if $\dim(X)=2n+1$ is odd, $A$ is the algebraic representative of $A^n_{\ZZ}(X)$, and the principal polarization
of $A$ is ``well behaved`` with respect to this regular isomorphism
we say that $A$ carries an \it incidence polarization. \rm

The definition of categorical representability could seem rather disjoint from the classical ones.
It is nevertheless clear that rational representability is strongly related to the structure of the motive of $X$.
For example, if $X$ is a threefold, then rational representability of all the $A^i_{\QQ}(X)$ is equivalent to the existence of
a specific Chow--K\"unneth decomposition \cite{gorch-gul-motives-and-repr}.
A first point to note is then that fully faithful functors should hold motivic maps, as stated in the following conjecture by Orlov.

\begin{conjecture}[\cite{orlov-motiv}]\label{orlov-conj}
Let $X$ and $Y$ be smooth projective varieties and $\Phi: \Db(Y) \to \Db(X)$ be a fully faithful functor. Then
the motive $h(Y)$ is a direct summand of the motive $h(X)$.
\end{conjecture}

In order to get a link between categorical and rational representability, we should consider the former in
dimension 1. Note that being categorically representable in dimension 1 is equivalent to the existence
of a semiorthogonal decomposition by exceptional objects and derived categories of curves.
Orlov conjecture would then imply that if $X$ is categorically representable in dimension 1, then
its motive is a finite sum of abelian and discrete motives, and this would give informations about rational representability
for $A^i_{\QQ}(X)$. Being categorically representable in dimension 1
seems to be in fact a very strong condition. For example a smooth cubic threefold is strongly representable
with incidence property but not categorically representable, otherwise we would have the splitting
of the intermediate Jacobian (see Corollary \ref{cor-no-split-no-fm}). 
Notice that in \cite{kuzne-manivel-marku} the study of the Abel--Jacobi map for some hypersurfaces and
its link with categorical constructions were already treated. 

Algebraic representability and the incidence property can have deep interactions with
categorical representability, and this is indeed the heart of the proof of Theorem 1.1 in \cite{marcello-bolo-conicbd}.
Consider a smooth projective threefold $X$ and assume it to be
rationally representable, with $h^1(X)=h^5(X)=0$, and with $A^2_{\ZZ}(X)$ algebraically representable with the incidence property. The arguments
in \cite{marcello-bolo-conicbd} show that if $X$ is categorically representable in dimension 1, then
the intermediate Jacobian $J(X)$ splits into Jacobians of curves, namely of those curves of positive genus appearing in the
semiorthogonal decomposition. This result can then be applied to a large class of complex threefolds with negative
Kodaira dimension (see a list in Remark \ref{remark-list-of-3folds}).

We can then reasonably raise the following question, which also points out how this new definition could be useful
in higher dimensions: is categorical representability in codimension 2 a necessary condition
for rationality? This is true for curves (where we have to replace codimension 2 with dimension 0) and
for surfaces, since any rational smooth projective surface admits a full exceptional sequence.
Remark \ref{list-of-iff} shows that this is true for a wide class of complex threefods with negative Kodaira dimension, but we can only
argue so far by a case by case analysis. Categorical representability should moreover hold the
vanishing of the Clemens--Griffiths component of $\Db(X)$ mentioned in \cite{kuznetcubicfourfold}.
We can wonder if Kuznetsov's conjecture about rationality of cubic fourfold (\cite[Conj. 1.1]{kuznetcubicfourfold}) could then
be restated as follows: a cubic fourfold is rational if and only if it is categorically representable in
codimension 2. Finally, we can argue some conjectural relation between
categorical representability and the existence of gaps in the Orlov spectrum defined
in \cite{katza-favero-ballard-generation-time}.

\subsection*{Notations}

Any triangulated category is assumed to be essentially small. 
Given a smooth projective variety $X$, we denote $\kappa_X$ its Kodaira dimension, $\Db(X)$ the bounded derived 
category of coherent sheaves on it, $K_0(X)$ its Grothendieck group, $CH_{\ZZ}^d(X)$ the Chow group of
codimension $d$ cycles, and $A_{\ZZ}^d(X)$ the subgroup of algebraically trivial cycles in $CH_{\ZZ}^d(X)$.
If is $X$ pure $d$-dimensional, and $Y$ any smooth projective variety,
we denote by $\corr^i(X,Y):= CH_{\QQ}^{i+d}(X\times Y)$ the group of correspondences with rational
coefficients. If $X=\coprod X_j$, with $X_j$ connected, then
$\corr^i(X,Y) = \oplus \corr^i(X_j,Y)$.

\section{Categorical and classical representabilities for smooth projective varieties} 

\subsection{Semiorthogonal decompositions and categorical representability}
We start by recalling some categorical definitions which are necessary to define representability.
Let $K$ be a field and $\cat{T}$ a $K$-linear triangulated category. A full triangulated category $\cat{A}$ of
$\cat{T}$ is called \it admissible \rm if the embedding functor admits a left and a right adjoint.

\begin{definition}[\cite{bondalkap,bondalorlov}]\label{def-semiortho}
A \it semiorthogonal decomposition \rm of $\cat{T}$ is a sequence of admissible subcategories
$\cat{A}_1, \ldots, \cat{A}_l$ of $\cat{T}$ such that $\Hom_{\cat{T}}(A_i,A_j) = 0$ for all $i>j$ and for 
all objects $A_i$ in $\cat{A}_i$ and $A_j$ in $\cat{A}_j$, and for every object $T$ of $\cat{T}$, there is a chain
of morphisms $0=T_n \to T_{n-1} \to \ldots \to T_1 \to T_0 = T$ such that the cone of $T_k \to T_{k-1}$ is
an object of $\cat{A}_k$ for all $k=1,\ldots,l$. Such a decomposition will be written
$$\cat{T} = \langle \cat{A}_1, \ldots, \cat{A}_l \rangle.$$
\end{definition}

\begin{definition}[\cite{bondal}]\label{def-except}
An object $E$ of $\cat{T}$ is called \it exceptional \rm if $\Hom_{\cat{T}} (E,E) = K$, and $\Hom_{\cat{T}}(E,E[i])=0$
for all $i \neq 0$. A collection $\{E_1,\ldots,E_l\}$ of exceptional objects is called \it exceptional \rm if
$\Hom_{\cat{T}}(E_j,E_k[i])=0$ for all $j>k$ and for all integer $i$.
\end{definition}

If $E$ in $\cat{T}$ is an exceptional object, the triangulated category generated by $E$ (that is, the smallest full
triangulated subcategory of $\cat{T}$ containing $E$) is equivalent to the derived category of a point, seen as a smooth projective
variety. The equivalence $\Db(pt) \to \langle E \rangle \subset \cat{T}$ is indeed given by sending $\ko_{pt}$ to $E$.
Given an exceptional collection $\{E_1,\ldots,E_l\}$ in the derived category $\Db(X)$ of a smooth projective variety,
there is a semiorthogonal decomposition \cite{bondalorlov}
$$\Db(X) = \langle \cat{A}, E_1, \ldots, E_l\rangle,$$
where $\cat{A}$ is the full triangulated subcategory whose objects are all the $A$ satisfying $\Hom(E_i,A)=0$
for all $i=1,\ldots,l,$ and we denote by $E_i$ the category generated by $E_i$.
We say that the exceptional sequence is \it full \rm if the category $\cat{A}$ is trivial.

\begin{definition}\label{def-rep-for-cat}
A triangulated category $\cat{T}$ is \it representable in dimension $m$ \rm
if it admits a semiorthogonal decomposition
$$\cat{T} = \langle \cat{A}_1, \ldots, \cat{A}_l \rangle,$$
and for all $i=1,\ldots,l$ there exists a smooth projective variety $Y_i$ with $\dim Y_i \leq m$, such that $\cat{A}_i$
is equivalent to an admissible subcategory of $\Db(Y_i)$.
\end{definition}

\begin{remark}\label{connectd}
Notice that we can assume that the categories $\cat{A}_i$ to be indecomposable, and
then the varieties $Y_i$ in the definition to be connected. Indeed, the derived category
$\Db(Y)$ of a scheme $Y$ is indecomposable if and only if $Y$ is connected (see \cite[Ex. 3.2]{bridg-equiv-and-FM}).
\end{remark}

\begin{definition}
Let $X$ be a smooth projective variety of dimension $n$. We say that $X$ is \it categorically representable \rm
in dimension $m$ (or equivalently in codimension $n-m$) if $\Db(X)$ is representable in dimension $m$.
\end{definition}

\begin{remark}\label{rem-def-for-non-smooth}
Suppose that $X$ is not smooth. Then to define categorical representability for it, we need to use
categorical resolution of singularities, as defined by Kuznetsov \cite{kuznet-singul}. He constructs,
provided that $X$ has rational singularities, and given a resolution $\widetilde{X} \to X$,
an admissible subcategory $\widetilde{\cat{D}}$ of $\Db(\widetilde{X})$ which is the \it categorical
resolution of singularities \rm of $\Db(X)$. Then we can say that $X$ is categorically representable
in dimension $m$ if $\widetilde{\cat{D}}$ is.
\end{remark}

Notice that if any fully faithful functor between smooth projective varieties is of Fourier--Mukai type \cite{orlovequivk3,orlovequivall}.
It is moreover worth noting and recalling the following facts, which are well-known in the derived categorical setting.

\begin{remark}[\cite{beilinson}]
The derived category of $\PP^n$ admits a full exceptional sequence.
\end{remark}

\begin{remark}
If $\Gamma$ is a smooth connected projective curve of positive genus, then $\Db(\Gamma)$ has no proper
admissible subcategory. Indeed any fully faithful functor
$\cat{A} \to \Db(\Gamma)$ is an equivalence, unless $\cat{A}$ is trivial. Then being categorically representable in dimension
1 is equivalent to admit a semiorthogonal decomposition by exceptional objects and derived
categories of smooth projective curves.
\end{remark}
\begin{remark}\label{rem-grr-decomp}
If $X$ and $Y_i$ are smooth projective and
$$\Db(X) = \langle \Db(Y_1), \ldots, \Db(Y_k) \rangle,$$
then
$$K_0(X) = \bigoplus_{i=1}^k K_0(Y_i)$$
and the Grothendieck--Riemann--Roch Theorem gives
$$\CH^*_{\QQ}(X) = \bigoplus_{i=1}^k \CH^*_{\QQ}(Y_i).$$
\end{remark}
\begin{proposition}[\cite{orlovprojbund}]\label{blow-up-formula}
Let $X$ be smooth projective and $Z \subset X$ a smooth subvariety of codimension $d > 1$. Denote by
$\varepsilon:\widetilde{X} \to X$ the blow up of $X$ along $Z$. Then
$$\Db(\widetilde{X}) = \langle \varepsilon^* \Db(X), \Db(Z)_1, \ldots, \Db(Z)_{d-1} \rangle,$$
where $\Db(Z)_i$ is equivalent to $\Db(Z)$ for all $i=1,\ldots,d-1$.
\end{proposition}

\subsection{Classical representabilities and motives}

In this Section we outline a list of definitions of representabilities for the groups $A^i_{\ZZ}(X)$. This is
far for being exhaustive, especially in the referencing.
Indeed, giving a faithful list of all contributions
to these questions is out of the aim of these notes. Chow motives and their properties could give, through Conjecture
\ref{orlov-conj}, a way to connect categorical and classical representabilities. We also outline the basic
facts needed to stress the possible interplay between new and old definitions.

Let $X$ as usual be a smooth projective variety over an algebraically closed field $K$.
\begin{definition}[\cite{blochmurreFano}]
The group $A_{\ZZ}^i(X)$ is said to be \it weakly representable \rm if there exists a smooth projective curve $\Gamma$, a class
$z$ of a cycle in $CH_{\ZZ}^i(X\times \Gamma)$ and an algebraic subgroup $G \subset J(\Gamma)$ of the Jacobian variety
of $\Gamma$, such that the induced morphism
$$z_*:J(\Gamma) \rightarrow A^i_{\ZZ}(X)$$
is surjective with kernel $G$.
\end{definition}

When working with coefficients in $\QQ$, we have the following definition.

\begin{definition}\label{rat}
The group $A^i_{\QQ}(X)$ is \it rationally representable \rm if  
there exists a regular surjective morphism
$$z_*:J_{\QQ}(\Gamma) \rightarrow A^i_{\QQ}(X).$$

\end{definition}

Rational representability is a name that has been used several times in the literature, so it might lead to some misunderstanding.
We underline that Definition \ref{rat} is exactly the one from (\cite{gorch-gul-motives-and-repr}, page 5).
In the complex case, we have also a
stronger notion, which is called the \it Abel--Jacobi property \rm \cite{blochmurreFano}, which requires the existence
of an isogeny (i.e. a regular surjective morphism) $A^i_{\ZZ}(X) \to J^i(X)$ onto the $i$-th intermediate Jacobian.
The Abel-Jacobi property implies weak representability for smooth projective varieties defined on $\CC$.


\begin{definition}[\cite{beauvilleprym}]
An abelian variety $A$ is said to be the algebraic representative of $A^i_{\ZZ}(X)$ if there exists an isomorphism
$G:A^i_{\ZZ}(X) \rightarrow A$ which is universal. That is: for any morphism $g$ from $A^i_{\ZZ}(X)$ to an abelian variety $B$,
there exists a unique morphism $u:A \rightarrow B$ such that $u\circ G=g$. In this case we say that $A^i_{\ZZ}(X)$ is
\it algebraically representable. \rm
\end{definition}

The first examples of algebraic representatives are the Picard variety $\Pic^0(X)$ or the Albanese variety
$\Alb(X)$ if $n=1$ or, respectively, $n=\dim(X)$.
 
\begin{definition}
Let $X$ be a smooth projective variety of odd dimension $2n + 1$
and $A$ the algebraic representative of $A_{\ZZ}^{n+1}(X)$ via the canonical map $G:
A_{\ZZ}^{n+1}(X)\rightarrow A$. A polarization of $A$ with associated correspondence $\theta_A$ in $\corr(A)$, is the \it incidence polarization \rm with respect 
to $X$ if for all algebraic maps $f : T \rightarrow A_{\ZZ}^{n+1}(X)$ defined by a cycle $z$ in $CH_{\ZZ}^{n+1}(X\times T)$, we have
$$(G\circ f)^*\theta_A = (-1)^{n+1}I(z);$$
where $I(z)$ in $\corr(T)$ is the composition of the correspondences $z \in \corr(T,X)$ and $z\in \corr(X,T)$.
\end{definition}

There are many complex threefolds $X$ with negative Kodaira dimension, for which $A^2_{\ZZ}(X)$ is strongly represented by a generalized
Prym with incidence polarization. For these threefolds, we will show how categorical representability in dimension
1 gives a splitting of the intermediate Jacobian. A list of the main cases will be given in Section \ref{section-reconstr}.

\smallskip

A more modern approach to representability questions has to take Chow motives into account.
Let us recall their basic definitions and notations.
The category $\km_K$ of Chow motives over $K$
with rational coefficients is defined as follows: an object of $\km_K$ is a triple $(X,p,m)$, where $X$
is a variety, $m$ an integer and $p \in \corr^0(X,X)$ an idempotent,
called a \it projector\rm. Morphisms from $(X,p,m)$ to $(Y,q,n)$ are given by elements of $\corr^{n-m}(X,Y)$
precomposed with $p$ and composed with $q$.

There is a natural functor $h$ from the category of smooth projective schemes to the category of motives,
defined by $h(X) = (X,\id,0)$, and, for any morphism $\phi: X \to Y$, $h(\phi)$ being the correspondence given
by the graph of $\phi$. We write $\mathbf{1}:=(\Spec \, K, \id, 0)$ for the unit motive and $\LL := (\Spec \, K, \id, -1)$
for the Lefschetz motive,
and $M(-i) := M \otimes \LL^i$. Moreover, we have
$\Hom(\LL^d,h(X))= CH_{\QQ}^d(X)$ for all smooth projective schemes $X$ and all integers $d$.

If $X$ is irreducible of dimension $d$, the embedding $\alpha: pt \to X$ of the point defines a motivic map $\mathbf{1} \to h(X)$. We denote
by $h^0(X)$ its image and by $h^{\geq 1}(X)$ the quotient of $h(X)$ via $h^0(X)$. Similarly, $\LL^d$ is a quotient of $h(X)$, 
and we denote it by $h^{2d}(X)$.

In the case of smooth projective curves of positive genus there exists another factor which corresponds to the Jacobian variety of the 
curve. Let $C$ be a smooth projective connected curve, let us define a motive $h^1(C)$ such that we have a direct sum:
$$h(C) = h^0(C) \oplus h^1(C) \oplus h^2(C).$$

The upshot is that the theory of the motives $h^1(C)$ corresponds to that of Jacobian varieties (up to isogeny), in fact we have
$$\Hom(h^1(C),h^1(C')) = \Hom(J(C),J(C'))\otimes \QQ.$$

In particular, the full subcategory of $\km_K$ whose objects are direct summands of the motive $h^1(C)$ is equivalent
to the category of abelian subvarieties of $J(C)$ up to isogeny. Such motives can be called \it abelian. \rm We will say
that a motive is \it discrete \rm if it is the direct sum of a finite number of Lefschetz motives.

The strict interplay between motives and representability for threefolds is shown by Gorchinskiy and Guletskii.
In this case, the rational representability of $A^i_{\QQ}(X)$ for $i \geq 2$ is known (\cite{murre-resultat}).
In \cite{gorch-gul-motives-and-repr} it is proved that
$A^3_{\QQ}(X)$ is rationally representable if and only if the Chow motive of $X$ has a given
Chow-K\"{u}nneth decomposition.

\begin{theorem}[\cite{gorch-gul-motives-and-repr}, Thm 8]\label{theo-gor-gul}
Let $X$ be a smooth projective threefold. The group $A^3_{\QQ}(X)$ is rationally representable if and only if the motive $h(X)$
has the following Chow-K\"{u}nneth decomposition:
$$h(X) \cong \mathbf{1} \oplus h^1(X) \oplus \LL^{\oplus b} \oplus (h^1(J)(-1)) \oplus
(\LL^2)^{\oplus b} \oplus h^5(X) \oplus \LL^3,$$
where $h^1(X)$ and $h^5(X)$ are the Picard and Albanese motives respectively,
$b = b^2(X) = b^4(X)$ is the Betti number, and $J$ is a certain abelian
variety, which is isogenous to the intermediate Jacobian $J(X)$ if $K = \CC$.
\end{theorem}

\section{Interactions between categorical and classical representabilities}\label{tre}

In this section, we will consider varieties defined over the complex numbers. This restriction
is not really necessary, since most of the constructions work over any algebraically closed
field. Anyway, in the complex case, we can simplify our treatment by dealing with
intermediate Jacobians. Moreover, it will be more simple to list examples without the need
to make the choice of the base field explicit for any case.

\subsection{Fully faithful functors and motives}\label{sec-fff-and-mot}

At the end of the last section we have seen that, in the case of threefolds, rational representability of $A^3_{\QQ}(X)$ is equivalent
to the existence of some Chow-K\"unneth decomposition. 
The first step in relating categorical and rational representability is exploiting an idea of Orlov about the
motivic decomposition which should be induced by a fully faithful functor between the derived categories of
smooth projective varieties. Assuming this conjecture we get that for threefolds categorical representability in
dimension 1 is a stronger notion than rational representability.

Let us sketch Orlov's idea \cite{orlov-motiv}. If $X$ and $Y$ are smooth projective varieties
of dimension respectively $n$ and $m$,
and $\Phi: \Db(Y) \to \Db(X)$ is a fully faithful functor, then it is of Fourier--Mukai type \cite{orlovequivk3,orlovequivall}.
Let $\ke$ in $\Db(X \times Y)$ be its kernel and $\kf$ in $\Db(X \times Y)$ the kernel
of its right adjoint $\Psi$, we have $\kf \simeq \ke^{\vee} \otimes pr_X^{*}\omega_X [\dim X]$ (see \cite{mukaiabelian}).
Consider $e:= ch(\ke) \Td(X)$ and $f:= ch(\kf) \Td(Y)$, two mixed rational cycles in $\CH^*_{\QQ}(X \times Y)$. We
denote by $e_i$ (resp. $f_i$) the $i$-th codimensional component of $e$ (resp. $f$), that is
$e_i, f_i \in \CH_{\QQ}^i (X \times Y)$. As correspondences they induce
motivic maps $e_i: h(Y) \to h(X)(i-n)$ and $f_j: h(X)(m-j) \to h(Y)$. The Grothendieck--Riemann--Roch
Theorem implies that $f.e := \bigoplus_{i=0}^{n+m} f_{n+m-i} e_i = \id_{h(Y)}$.

\begin{orlconj}
Let $X$ and $Y$ be smooth projective varieties and $\Phi: \Db(Y) \to \Db(X)$ be a fully faithful functor. Then
the motive $h(Y)$ is a direct summand of the motive $h(X)$.
\end{orlconj}

The Conjecture is trivially true for $Y$ a smooth point, in which case $\Phi(\Db(Y))$ is generated
by an exceptional object of $\Db(X)$.
In \cite{orlov-motiv}, it is proven that the Conjecture holds if $X$ and $Y$ have the same dimension $n$
and $\ke$ is supported in dimension $n$. This already covers some interesting example: if $X$ is a smooth
blow-up of $Y$, or if there is a standard flip from $X$ to $Y$. Using the same methods as in \cite{marcello-bolo-conicbd}
we will show that the conjecture holds (up to restricting to all direct summand of $h(Y)$) for some more examples, namely the 
case of $Y$ a curve and $X$ a rationally representable threefold (i.e., $A^i_{\QQ}(X)$ is rationally representable for
all $i$) with $h^1(X) = h^5(X) = 0$.

But let us first take a look to the simplest case, that is relating categorical representability in dimension 0 and
discreteness of the motive.
\begin{remark}\label{prop-categorical-repr-in-dim-0}
If a smooth projective variety $X$ is categorically representable in dimension $0$, then the motive
$h(X)$ is discrete.
\end{remark}
\begin{proof}
Being representable in dimension $0$ is equivalent to having a full exceptional sequence. Then the
proof is straightforward, we actually have more, that is $K(X) = \ZZ^l$, where $l$ is the number of
objects in the sequence.
\end{proof}

A way more interesting case relates categorical representability in dimension 1 and rational representability
for threefolds.
In this case, in light of Theorem \ref{theo-gor-gul}, we have a more specific conjecture.
\begin{conjecture}
If a smooth projective threefold $X$ is categorically representable in dimension $1$, then it is rationally
representable.
\end{conjecture}

If $X$ is a standard conic bundle over a rational surface and $\Gamma$ a smooth projective curve, the Chow-K\"unneth decomposition of $h(X)$
(see \cite{nagel-saito}) can be used to show that a fully faithful functor $\Db(\Gamma) \to
\Db(X)$ gives $h^1(\Gamma)(-1)$ as a direct summand of $h(X)$. In particular, this
gives an isogeny between $J(\Gamma)$ and an abelian subvariety of $J(X)$, and proves
(up to codimensional shfit for each direct summand of $h(\Gamma)$) Conjecture \ref{orlov-conj}
in this case. The proof in \cite{marcello-bolo-conicbd} is based on the fact that
the motive $h(X)$ splits into a discrete motive and in a unique abelian motive which
corresponds to $J(X)$. Let us make a first assumption
\begin{itemize}
 \item[($\star$)] $X$ is a smooth projective rationally representable threefold with $h^1(X) = 0$ and $h^5(X) = 0$. 
\end{itemize}

\begin{theorem}\label{from-Db-to-isogeny}
Suppose $X$ satisfies $\star$. If there is a smooth projective curve $\Gamma$ and
a fully faithful functor $\Db(\Gamma) \to \Db(X)$, then there exists an integer $j_i$ such that $h^i(\Gamma)(j_i)$ is
a direct summand of $h(X)$ for $i=0,1,2$, and there is an injective
morphism $J(\Gamma)_{\QQ} \to J(X)_{\QQ}$, that is an isogeny between $J(\Gamma)$ and an abelian subvariety
of $J(X)$. 
\end{theorem}
\begin{proof}
Let $\ke$ and $\kf$, and $e$ and $f$ as before. Consider $h^0(\Gamma)= \mathbf{1}$, we have
$f.e_{\vert h^0(\Gamma)} = \id_{h^0(\Gamma)}$, which gives the claim, but not an
explicit value of $i_0$. The same argument works for $h^2(\Gamma) = \LL$.

For $h^1(\Gamma)$, we only need the case where $g(\Gamma) >0$, and
we can use the same argument as in \cite{marcello-bolo-conicbd} Lemma 4.2 : since
all but one addendum of $h(X)$ are discrete, the map $f.e_{\vert h^1(\Gamma)} = \id_{h^1(\Gamma)}$ is given
by $f_2.e_2$, which proves that $h^1(\Gamma)(-1)$ is a direct summand of $h^3(X)=M^1(J)(-1)$. 
\end{proof}

\begin{corollary}\label{corollary-cat-rep-and-isogeny}
Suppose $X$ satisfies $\star$ and let $\{\Gamma_i\}_{i=1}^k$ be smooth projective curves
of positive genus. If $\Db(X)$ is categorically representable
in dimension 1 by the categories $\Db(\Gamma_i)$ and by exceptional objects, then $J(X)$ is isogenous
to $\oplus_{i=1}^k J(\Gamma_i)$.
\end{corollary}
\begin{proof}
Use remark \ref{rem-grr-decomp}.
\end{proof}

\begin{remark}\bf{(Threefolds satisfying $\star$).} \rm
By \cite{gorch-gul-motives-and-repr, nagel-saito} Fano threefolds, threefolds fibered in Del Pezzo or Enriques surfaces over $\PP^1$
with discrete Picard group,
and standard conic bundles over rational surfaces satisfy assumptions of Theorem \ref{from-Db-to-isogeny}.
\end{remark}

\subsection{Reconstruction of the intermediate Jacobian}\label{section-reconstr}

The aim of this section is to show how, under appropriate hypothesis, 
categorical representability in dimension 1 for a threefold $X$ gives a splitting of the intermediate Jacobian $J(X)$.
Notice that in the case of curves the derived category carries the information about the principal polarization
of the Jacobian \cite{marcellocurves}.
In the case of threefolds, we need first of all the hypothesis of Theorem \ref{from-Db-to-isogeny}.
As we will see, the crucial hypothesis that will allow us to recover also the principal polarization
is that the polarization on $J(X)$ is an \textit{incidence polarization}.

\begin{itemize}
 \item[($\natural$)] $X$ is a smooth projective rationally and algebraically representable threefold with $h^1(X) = 0$ and $h^5(X) = 0$
and the algebraic representative of $A^2_{\ZZ}(X)$ carries an incidence polarization.
\end{itemize}

\begin{theorem}\label{reconstr-of-interm}
Suppose $X$ satisfies $\natural$. Let
$\Gamma$ be smooth projective curve and $\Db(\Gamma) \to \Db(X)$ fully faithful. Then there is an injective
morphism $J(\Gamma) \hookrightarrow J(X)$ preserving the principal polarization, that is $J(X) = J(\Gamma) \oplus A$ for
some principally polarized abelian variety $A$.
\end{theorem}

\begin{proof}
From Theorem \ref{from-Db-to-isogeny} we get an isogeny. As in the proof of \cite[Prop. 4.4]{marcello-bolo-conicbd},
the incidence property shows that this isogeny is an injective morphism respecting the principal polarizations.
\end{proof}

\begin{corollary}\label{ortjac}
Suppose $X$ satisfies $\natural$ and let $\{\Gamma_i\}_{i=1}^k$ be smooth projective curves
of positive genus. If $\Db(X)$ is categorically representable
in dimension 1 by the categories $\Db(\Gamma_i)$ and by exceptional objects, then $J(X)$ is isomorphic
to $\oplus_{i=1}^k J(\Gamma_i)$ as principally polarized variety.
\end{corollary}

\begin{remark}\label{remark-list-of-3folds}\bf{(Threefolds satisfying $\natural$).} \rm
The assumptions of Theorem \ref{reconstr-of-interm} seem rather restrictive. Anyway, they are satisfied by
a quite big class of smooth projective threefolds with $\kappa_X < 0$. The Chow-K\"unneth decomposition
for the listed varieties
is provided by \cite{nagel-saito} for conic bundles and by \cite{gorch-gul-motives-and-repr} in any other case.
In the following list the references
point out the most general results about strong representability and incidence property.
Giving an exhaustive list of all the results and contributors would be out of reach
(already in the cubic threefold case). We will consider Fano threefolds with Picard number one only. The
interested reader could find an exhaustive treatment in \cite{isko-prok-fano}.
\begin{itemize}

\item[1)] Fano of index $>2$: $X$ is either $\PP^3$ or a smooth quadric.

\item[2)] Fano of index $2$: $X$ is a quartic double solid \cite{tihoquarticsolid} , or a smooth cubic in $\PP^4$
\cite{clemensgriffiths}, or an intersection of two quadrics in $\PP^5$ \cite{reidphd}, or a $V_5$ (in the last case $J(X)$ is trivial).

\item[3)] Fano of index $1$: $X$ is a general sextic double solid \cite{ceresaverra}, or
a smooth quartic in $\PP^4$ \cite{blochmurreFano},
or an intersection of a cubic and a quadric in $\PP^5$ \cite{blochmurreFano}, or the intersection of three
quadrics in $\PP^6$ \cite{beauvilleprym}, or a $V_{10}$ \cite{logachevV10,iliev10}, or a $V_{12}$ \cite{ilievmarkuV12} ($J(X)$
is the jacobian of a genus 7 curve), 
or a $V_{14}$  \cite{iliev-marku-v14} (in which case the representability is related to the birational map to a smooth cubic threefold),
or a general $V_{16}$ \cite{ilievV16, mukaigranmisto}, or a general $V_{18}$ \cite{ilievamanovella,isko-prok-fano}
($J(X)$ is the jacobian of a genus 2 curve), or a $V_{22}$ (and the Jacobian is trivial).

\item[4)] Conic bundles: $X \to S$ is a standard conic bundle over a rational surface \cite{beauvilleprym,beltrachow}, this is the case examined in
\cite{marcello-bolo-conicbd}.

\item[5)] Del Pezzo fibrations: $X \to \PP^1$ is a Del Pezzo fibration with $2\leq K_X^2 \leq 5$ \cite{kanevdp1,kanevdp2}.
\end{itemize}

\end{remark}
From the unicity of the splitting of the intermediate Jacobian we can easily infer the following.

\begin{corollary}
Suppose $X$ satisfies $\natural$ and is categorically representable in dimension 1, with semiorthogonal decomposition
$$\Db(X)= \langle \Db(\Gamma_1),\ldots,\Db(\Gamma_k),E_1,\ldots,E_l\rangle.$$
Then there is no fully faithful functor $\Db(\Gamma) \to \Db(X)$ unless $\Gamma \simeq \Gamma_i$ for some $i \in \{ 1,\ldots,k\}$.
Moreover, the semiorthogonal decomposition is essentially unique, that
is any semiorthogonal decomposition of $\Db(X)$ by smooth projective curves and exceptional objects is
given by all and only the curves $\Gamma_i$ and $l$ exceptional objects.
\end{corollary}

\begin{corollary}\label{cor-no-split-no-fm}
Suppose $X$ satisfies $\natural$, $\Gamma$ is a smooth projective curve of positive genus
and there is no splitting $J(X) = J(\Gamma) \oplus A$. Then there is no fully faithful functor
$\Db(\Gamma) \to \Db(X)$.
\end{corollary}
The assumptions of Corollary \ref{cor-no-split-no-fm} are trivially satisfied if the
threefold satisfying $\natural$ has $J(X)=0$. A way more interesting case is when
the intermediate Jacobian is not trivial and has no splitting at all, in which case the
variety is not representable in dimension $<2$.

\begin{remark}\label{list-three-not-rep}\bf{(Threefolds not categorically representable in dimension $<2$)}. \rm
The assumptions of Corollary \ref{cor-no-split-no-fm} are satisfied by smooth threefolds with $J(X) \neq 0$
for all curve $\Gamma$ of positive genus in the following cases:
\smallskip

Either $X$ is a smooth cubic \cite{clemensgriffiths}, or a generic quartic threefold \cite{letiziamoratti}, or a generic complete intersection of type $(3,2)$
in $\PP^5$ \cite{beauvilleprym} or a symmetric one \cite{bovesym}, or the intersection of three quadrics in $\PP^7$ \cite{beauvilleprym},
or a standard conic bundle $X \to \PP^2$ degenerating along a curve
of degree $\geq 6$ \cite{beauvilleprym}, or a non-rational standard conic bundle $X \to S$ on a Hirzebruch surface
\cite{shokuprym}, or a non-rational Del Pezzo fibration $X \to \PP^1$ of degree four \cite{aleks-dp4}.

There are some other cases of Fano threefolds of specific type satisfying geometric assumptions. For a detailed treatment, see
\cite[Chapt. 8]{isko-prok-fano}.
\end{remark}

Notice that if $X$ is a smooth cubic threefold, the equivalence class of a notable admissible subcategory $\cat{A}_X$
(the orthogonal complement of $\{\ko_X,\ko_X(1)\}$) corresponds to the isomorphism class of $J(X)$ as principally polarized
abelian variety \cite{noicubic}; the proof is based on the reconstruction of the Fano variety and the techniques used
there are far away from the subject of this paper.

A natural question is if, under some hypothesis, one can give the inverse statement of Corollaries \ref{corollary-cat-rep-and-isogeny}
and \ref{ortjac}, that is: suppose that $X$ is a threefold satisfying either $\star$ or $\natural$, such
that $J(X) \simeq \oplus J(\Gamma_i)$. Can one describe a semiorthogonal decomposition of $\Db(X)$ by
exceptional objects and the categories $\Db(\Gamma_i)$? Notice that a positive
answer for $X$ implies a positive answer for all the smooth blow-ups of $X$.

\begin{remark}\label{list-of-iff}\bf(Threefolds with $\kappa_X <0$ categorically representable in dimension $\leq 1$). \rm
Let $X$ be a threefold satisfying $\star$ or $\natural$ and with $J(X) = \oplus J(\Gamma_i)$. Then if $X$
is in the following list (or is obtained by a finite number of smooth blow-ups from a variety in the list)
we have a semiorthogonal decomposition
$$\Db(X) = \langle \Db(\Gamma_1), \ldots, \Db(\Gamma_k), E_1, \ldots, E_l \rangle,$$
with $E_i$ exceptional objects.
\begin{itemize}
\item[1)] Threefolds with a full exceptional sequence: $X$ is $\PP^3$ \cite{beilinson},
or a smooth quadric \cite{kapranovquadric},
or a $\PP^1$-bundle over a rational surface or a $\PP^2$-bundle over $\PP^1$ \cite{orlovprojbund},
or a $V_5$ \cite{orlov-v5}, or a $V_{22}$ Fano threefold \cite{kuznev22}. 
\item[2)] Fano threefolds without any full exceptional sequence: $X$ is the complete intersection of two quadrics or a Fano
threefold of type $V_{18}$, and $J(\Gamma) \simeq J(X)$ with $\Gamma$ hyperelliptic.
The semiorthogonal decompositions are described in \cite{bondalorlov,kuznetHyperplane}, and are strikingly
related (as in the cases of $V_5$ and $V_{22}$ and of the cubic and $V_{14}$) by a correspondence in
the moduli spaces, as described in \cite{kuznetfanothreefolds}. $X$ is a $V_{12}$ Fano threefold
\cite{kuznetv12}, or a $V_{16}$ Fano threefold \cite{kuznetHyperplane}.
\item[3)] Conic bundles without any full exceptional sequence: $X \to S$ is a rational conic bundle over a minimal surface
\cite{marcello-bolo-conicbd}. If the degeneration locus of $X$ is either empty or a cubic in $\PP^2$, then $X$ is a
$\PP^1$-bundle and is listed in 1).
\item[4)] Del Pezzo fibrations: $X \to \PP^1$ is a quadric fibration with at most simple degenerations, in which case the
hyperelliptic curve $\Gamma \to \PP^1$ ramified along the degeneration appears naturally as the orthogonal complement of
an exceptional sequence of $\Db(X)$ \cite{kuznetconicbundles}. $X \to \PP^1$ is a rational Del Pezzo fibration of degree four.
In this case $X$ is birational to a conic bundle over a Hirzebruch surface
\cite{aleks-dp4} and the semiorthogonal decomposition will be described in a forthcoming paper \cite{asher-marcello-bolo}.
\end{itemize}
Notice that the first two items cover all classes of Fano threefolds with Picard number one whose members are all rational.
\end{remark}

\section[Developments and Questions]{Categorical representability and rationality: \\
further developments and open questions}

This last Section is dedicated to speculations and open questions about categorical representability and rationality.
The baby example of curves is full understood. A smooth projective curve $X$ over a field $K$ is categorically representable in
dimension 0 if and only if it is rational. Indeed, the only case where $\Db(X)$ has exceptional objects is
$X=\PP^1$, and $\Db(X) = \langle \ko_X,\ko_X(1) \rangle$.

Let us start with a trivial remark: the projective space $\PP^n$ over $K$ is categorically representable in dimension 0.
Then if $X$ is given by a finite number of smooth blow-ups of $\PP^n$, it is categorically representable in
codimension $\geq 2$. This is easily obtained from Orlov's blow-up formula (see Prop. \ref{blow-up-formula}).
More generally, if a smooth projective variety $X$ of dimension $\geq 2$ is categorically representable in
codimension $m$, then any finite chain of smooth blow-ups of $X$ is categorically representable in
codimension $\geq {\rm min}(2,m)$.

One could naively wonder about the inverse statement: if $X \to Y$
is a finite chain of smooth blow-ups and $X$ is categorically representable in codimension $m$, what can
we say about $Y$? Unfortunately, triangulated categories do not have enough structure to let us
compare different semiorthogonal decomposition. For example, the theory of mutations allows to do this only
in a few very special cases.

In this Section we present some more example to stress how the interaction between categorical representability
and rationality can be devloped further, and we point out some open question.
We deal with surfaces in \ref{surf} and with threefolds in \ref{threefolds}.
Then we will discuss in \ref{noncommutative} how categorical
representability for noncommutative varieties plays an important role in this frame, to deal with
varieties of dimension bigger than 3 in \ref{higherdim}. Finally, we compare in \ref{approaches}
our methods with recent approaches to birationality
problems via derived categories. We will work over the field $\CC$ for simplicity, even if many
problems and arguments do not depend on that.

\subsection{Surfaces}\label{surf}

If $X$ is a smooth projective rational surface,
then it is categorically representable in codimension 2. Indeed,
$X$ is the blow-up in a finite number of smooth points of a minimal rational surface, that is either $\PP^2$ or
$\FF_n$. Are there any other example of surfaces categorically representable in codimension 2? Notice that by
Proposition \ref{prop-categorical-repr-in-dim-0} such a surface would have a discrete motive, and even more:
we would have $K_0(X) = \ZZ^l$. In particular, if $K_0(X)$ is not locally free, then
$X$ is not categorically representable in dimension 0.

In general, an interesting problem is to construct exceptional sequences on surfaces with $p_g=q=0$, and to study their
orthogonal complement. Suppose for example that $X$ is an Enriques surface: a (non-full) exceptional collection of 10
vector bundles on $X$ is described in \cite{zube-enriques}. Since $K_0(X)$ is not locally free, we do not expect any
full exceptional collection. The orthogonal complement $\cat{A}_X$ turns then out to be a very interesting
object, related also to the geometry of some singular quartic double solid \cite{kuzne-ingalls}.
Using a motivic trick, we can prove that, under some assumption, a surface with $p_g=q=0$ is
either categorically representable in codimension 2 or not categorically representable in positive codimension.

\begin{proposition}\label{prop-surf-p=q=0}
Let $X$ be a surface with $h(X)$ discrete. Then for any curve $\Gamma$ of positive genus, there is no fully faithful
functor $\Db(\Gamma) \to \Db(X)$.
\end{proposition}
\begin{proof}
Suppose there is such a curve and such a functor $\Phi: \Db(\Gamma) \to \Db(X)$. Let $\ke$ denote the kernel of
$\Phi$ (which has to be of Fourier--Mukai type) and $\kf$ the kernel of its adjoint. Consider the cycles $e$ and $f$ described
in Section \ref{sec-fff-and-mot}, and recall that $f.e = \oplus_{i=0}^3 f_{3-i}.e_i = \id_{h(\Gamma)}$. Restricting now
to $h^1(\Gamma)$ we would have that $\id_{h^1(\Gamma)}$ would factor through a discrete motive, which is impossible.
\end{proof}

\begin{corollary}\label{cor-surf-non-rep}
Let $X$ be a surface with $h(X)$ discrete and $K_0(X)$ not locally free. Then $X$ is not categorically representable in codimension
$>0$.
\end{corollary}

\begin{remark}\bf(Surfaces with $p_g=q=0$ not categorically representable in positive codimension). \rm
Proposition \ref{prop-categorical-repr-in-dim-0} could be an interesting tool in the study of derived categories of surfaces
with $p_g=q=0$: notice that many of them have torsion in $H_1(X, \ZZ)$ (for an exhaustive treatment and referencing, see
\cite{bauer-cata-pigna-survey}). Anyway the discreteness of the motive is a rather strong assumption, which for
example implies the Bloch conjecture. There are few cases where this is known.
\begin{itemize}
 \item[1)] $X$ is an Enriques surface \cite{coombes}. 
\item[2)] $X$ is a Godeaux surface obtained as a quotient of a quintic by an action of $\ZZ/5\ZZ$ \cite{gul-pedr-godeaux}.
\end{itemize}
\end{remark}

These observations lead to state some deep question about categorical representability of
surfaces.

\begin{question}\label{exc-obj-on-surf}
Let $X$ be a smooth projective surface with $p_g = q = 0$. 
\begin{itemize}
\item[1)] Does $\Db(X)$ admit an exceptional sequence?

\item[2)] Is the exceptional sequence full? That is, is $X$ representable in codimension 2?

\item[3)] If $X$ is representable in codimension 2, is $X$ rational?
\end{itemize}
\end{question}

\subsection{Threefolds}\label{threefolds}

Remark that there are examples
of smooth projective non-rational threefolds $X$ which are categorically representable in codimension 2:
just consider a rank three vector bundle $\ke$ on a curve $C$ of positive genus and take $X:=\PP(\ke)$.
In \cite[Sect. 6.3]{marcello-bolo-conicbd} we provide a conic bundle example. Anyway, Corollary
\ref{ortjac} somehow suggests that categorical representability in codimension 2 should be a necessary condition
for rationality.

A reasonable idea is to restrict our
attention to minimal threefolds with $\kappa_X <0$ (recall that this is a necessary condition for rationality),
in particular to the ones we expect to satisfy assumption $\natural$, in order to have interesting information about the intermediate Jacobian from semiorthogonal decompositions.
The three big families of such threefolds are:
Fano threefolds, conic bundles over rational surfaces and del Pezzo fibrations over $\PP^1$.
Remark \ref{list-of-iff} gives a list of rational threefolds which are categorically representable in codimension 2,
and Remark \ref{remark-list-of-3folds} a list of families whose generic term is non-rational and cannot be categorically
representable in codimension $> 1$.

\begin{question}\label{question-on-3folds}
Let $X$ be a smooth projective threefold with $\kappa_X <0$.

\begin{itemize}
\item[1)] If $X$ is rational, is $X$ categorically representable in codimension 2?

\item[2)] Is $X$ categorically representable in codimension 2 if and only if $X$ is rational? 
\end{itemize}
\end{question}

A positive answer to the second question is provided for standard conic bundles over minimal surfaces \cite{marcello-bolo-conicbd},
but it seems to be quite a strong fact to hold in general: recall that having a splitting $J(X) \simeq \oplus J(\Gamma_i)$ is only a necessary
condition for rationality, and Corollary \ref{ortjac} shows that if $X$ satisfies $\natural$, categorical representability
in codimension 2 would give the splitting of the Jacobian.
 
Remark \ref{list-of-iff} provides a large list of rational threefolds categorically representable in codimension 2.
Is it possible to add examples to this list? In particular in the case of Del Pezzo fibrations over $\PP^1$
only the quadric and the degree 4 fibration are described.

A good way to understand these questions is by studying some special rational or non-rational (that is non generic in their
family) threefold. This forces to consider non smooth ones, but we can use Kuznetsov's theory of categorical resolution of singularities \cite{kuznet-singul}
and study the categorical resolution of $\Db(X)$, as we pointed out in Remark \ref{rem-def-for-non-smooth}.
For example, let $X \subset \PP^4$ be nodal cubic threefold with a double point, which is known to be rational.
\begin{proposition}
Let $X \subset \PP^4$ be a cubic threefold with a double point and $\widetilde{X} \to X$ the blow-up of the singular point. The categorical
resolution of singularities $\widetilde{\cat{D}} \subset \Db(\widetilde{X})$ of $\Db(X)$ is representable in codimension two.
Indeed there is a semiorthogonal decomposition
$$\widetilde{\cat{D}} = \langle \Db(\Gamma), E \rangle,$$
where $E$ is an exceptional object and $\Gamma$ a complete intersection of a quadric and a cubic in $\PP^3$.
 \end{proposition}
\begin{proof}
This is shown following step by step \cite[Section 5]{kuznetcubicfourfold}, where the four dimensional case is studied.
Let us give a sketch of the proof. Let $P$ be the singular point of $X$, and $\sigma: \widetilde{X} \to X$ its blow-up.
The exceptional locus $\alpha: Q \hookrightarrow \widetilde{X}$ of $\sigma$ is a quadric surface.
The projection of $\PP^4$ to $\PP^3$ from the point $P$ restricted to $X$ gives the birational map $X \dashrightarrow \PP^3$.
The induced map $\pi: \widetilde{X} \to \PP^3$ is the blow-up of a smooth curve $\Gamma$ of genus 4, given by the complete intersection
of a cubic and a quadric surface.
If we denote $h:= \pi^* \ko_{\PP^3}(1)$ and $H:= \sigma^* \ko_X (1)$, we have that $Q=2h-D$, $H=3h-D$, then $h=H-Q$ and $D=2H-3Q$
as in \cite[Lemma 5.1]{kuznetcubicfourfold}. The canonical bundle $\omega_{\widetilde{X}} = -4h+D= -2H+Q$ can be calculated
via the blow-up $\pi$.

The same arguments as in \cite{kuznetcubicfourfold} give the decomposition
$$\Db(\widetilde{X}) = \langle \alpha_* \ko_Q (-h), \widetilde{\cat{D}} \rangle.$$
Indeed the conormal bundle of $Q$ is $\ko_Q(h)$ and the Lefschetz decomposition with respect to it is:
$$\langle \ka_1(-h),\ka_0\rangle,$$
where $\ka_1 = \langle\ko_Q\rangle$ and $\ka_0 = \langle \ko_Q, S_1, S_2 \rangle$, with $S_1$ and $S_2$ the two spinor bundles.
We obtain then the semiorthogonal decomposition:
\begin{equation}\label{semiorth1}
\Db(\widetilde{X}) = \langle \alpha_* \ko_Q (-h), \widetilde{\cat{A}}_X, \ko_{\widetilde{X}}, H \rangle,
\end{equation}
where $\widetilde{\cat{A}}_X$ is the categorical resolution of $\cat{A}_X$, as in \cite[Lemma 5.8]{kuznetcubicfourfold}.
The representability of $\widetilde{\cat{D}}$ relies then on the representability of $\widetilde{\cat{A}}_X$.

On the other side, applying the blow-up formula (see Prop. \ref{blow-up-formula}) to $\pi: \widetilde{X} \to \PP^3$,
and choosing $\{ \ko_{\PP^3}(-3), \ldots, \ko_{\PP^3} \}$ as full exceptional sequence for $\Db(\PP^3)$, we obtain:
$$\Db(\widetilde{X}) = \langle \Phi \Db(\Gamma), -3h,-2h-h,\ko_{\widetilde{X}} \rangle,$$
where $\Phi: \Db(\Gamma) \to \Db(\widetilde{X})$ is fully faithful. Now as in \cite[Lemma 5.3]{kuznetcubicfourfold},
if we mutate $-3h$ and $-2h$ to the left with respect to $\Phi \Db(\Gamma)$, we get
\begin{equation}\label{semiorth2}
\Db(\widetilde{X}) = \langle -3h+D, -2h+D, \cat{B}, \ko_{\widetilde{X}} \rangle,
\end{equation}
where $\cat{B} = \langle \Phi \Db(\Gamma), -h \rangle$ is an admissible subcategory of $\Db(\widetilde{X})$.
Finally, one can show that $\cat{B}$ and $\widetilde{\cat{A}}_X$ are equivalent,
following exactly the same path of mutations as in \cite[Sect. 5]{kuznetcubicfourfold} to compare the
decompositions (\ref{semiorth1}) and (\ref{semiorth2}). Notice that one can calculate explicitely the
exceptional object $E$ by following the mutations of $-h$.
\end{proof}

Another special very interesting example is described in \cite{kuzne-ingalls}: a singular double solid $X \to \PP^3$
ramified along a quartic symmetroid. This threefold is non-rational thanks to \cite{artin-mumford}, because $H^3(X, \ZZ)$
has torsion. A rough account
(skipping the details about the resolution of singularities) of Ingalls and Kuznetsov's result is the following:
if $X'$ is the small resolution of $X$, there is an Enriques surface $S$ and a semiorthogonal decomposition
\begin{equation}\label{decomp-double-solid}
\Db(X') = \langle \cat{A}_S, E_1,E_2 \rangle,
\end{equation}
where $E_i$ are exceptional objects and $\cat{A}_S$ is the orthogonal complement in $\Db(S)$ of 10 exceptional vector bundles on $S$
(\cite{zube-enriques}). Then we can apply to this set Corollary  \ref{cor-surf-non-rep}.
\begin{corollary}
The threefold $X'$ is not categorically representable in codimension $>1$.
\end{corollary}
\begin{proof}
Consider the Enriques surface $S$ and the semiorthogonal decomposition
$$\Db(S) = \langle \cat{A}_S, E_1, \ldots, E_{10} \rangle,$$
where $E_i$ are the exceptional vector bundles described in \cite{zube-enriques}. Then the non-representability
of $S$ in codimension $>0$ is equivalent to the non-representability of $\cat{A}_S$ in dimension $<2$.
The statement follows then from (\ref{decomp-double-solid}).
\end{proof}

Remark that the
lack of categorical representability of $X'$ (and presumably of $X$, thinking about the categorical resolution of
singularities) is due to the presence of torsion in $K_0(S)$ and in particular in $H_1(S,\ZZ)$, whereas the non-rationality
of $X$ is due to the presence of torsion in $H^3(X,\ZZ)$. The relation between torsion in $H^3(X,\ZZ)$ and categorical
representability needs a further investigation, for example in the case recently described in \cite{katza-iliev-pry-nonrational}.

\subsection{Noncommutative varieties}\label{noncommutative}

The previous speculations and partial results give rise to the hope of extending fruitfully the study of categorical representability
to higher dimensions and to the noncommutative setting. By the latter, we mean, following Kuznetsov
\cite[Sect. 2]{kuznetconicbundles}, an algebraic variety $Y$ with a sheaf $\kb$ of $\ko_Y$-algebras of finite type.
Very roughly, the corresponding noncommutative variety $\bar{Y}$ would have a category of coherent sheaves $\coh(\bar{Y}) = \coh(Y,
\kb)$ and a bounded derived category $\Db(\bar{Y}) = \Db(Y,\kb)$. The examples
which appear very naturally in our setting
are the cases where $\kb$ is an Azumaya algebra or the even part of the Clifford algebra associated to some quadratic form over $Y$.
Finally, if a triangulated category $\cat{A}$ has
Serre functor such that $S_{\cat{A}}^m = [n]$, for some integers $n$ and $m$, with $m$ minimal with this property,
we will call it a \it $\frac{n}{m}$-Calabi--Yau category. \rm
If $m=1$, these categories deserve the name of noncommutative Calabi--Yau $n$-folds, even if they are not a priori given
by the derived category of some Calabi--Yau $n$-fold with a sheaf of algebras.

If $S$ is any smooth projective variety, $X \to S$ a Brauer--Severi variety of relative dimension $r$, and $\ka$ the
associated Azumaya algebra in $\Br(S)$, then \cite{marcellobrauer}
$$\Db(X) = \langle \Db(S), \Db(S,\ka^{-1}), \ldots, \Db(S,\ka^{-r}) \rangle.$$
The categorical representability of $X$ would then
rely on the categorical representability of $(S,\ka)$, which is an interesting object in itself. For example,
if $Y$ is a generic cubic fourfold containing a plane, there are a K3 surface $S$ and an Azumaya algebra
$\ka$ such that the categorical representability of $(S,\ka)$ is the subject of Kuznetsov's Conjecture \cite{kuznetcubicfourfold} about the rationality of cubic fourfolds.

\smallskip

If $S$ is a smooth projective variety and $Q \to S$ a quadric fibration of relative dimension $r$, we
can consider the sheaf $\kb_0$ of the even parts of the Clifford algebra associated to the quadratic form defining
$Q$. There is a semiorthogonal decomposition:
$$\Db(Q) = \langle \Db(S,\kb_0), \Db(S)_1, \ldots, \Db(S)_{r-1}\rangle,$$
where $\Db(S)_i$ are equivalent to $\Db(S)$ \cite{kuznetconicbundles}. The categorical representability of $(S,\kb_0)$ should then 
be a very important tool in studying birational properties of $Q$. This is indeed the case for conic bundles over rational
surfaces \cite{marcello-bolo-conicbd}.

\smallskip

Finally, let $\cat{A}$ be an $\frac{n}{m}$-Calabi--Yau category. Such categories appear as orthogonal complements
of an exceptional sequence on Fano hypersurfaces in projective spaces \cite[Cor. 4.3]{kuznetv14}. It is then natural
to wonder about their representability. For example, if $X$ is a cubic or a quartic threefold,
it follows from Remark \ref{list-three-not-rep} that these orthogonal complements (which are, respectively, $\frac{5}{3}$ and
$\frac{10}{4}$-Calabi--Yau)
are not representable in dimension 1.

\begin{question}\label{question-cycat}
Let $\cat{A}$ be a $\frac{n}{m}$-Calabi--Yau category.
\begin{itemize}
\item[1)] Is $\cat{A}$ representable in some dimension?
\item[2)] If yes, is there an explicit lower bound for this dimension?
\item[3)] If $m=1$, is $\cat{A}$ representable in dimension $n$ if and only if there exist a smooth
$n$-dimensional variety $X$ and an equivalence $\Db(X) \simeq \cat{A}$?
\end{itemize}
\end{question}

\subsection{Higher dimensional varieties}\label{higherdim}
Unfortunately, it looks like the techniques used for threefolds in \cite{marcello-bolo-conicbd} hardly extend to
dimension bigger than $3$. The examples and supporting evidences provided so far lead anyway
to suppose that categorical representability can give useful informations on the birational properties of 
any projective variety. The main case is a challenging Conjecture by Kuznetsov \cite{kuznetcubicfourfold}. Let $X \subset \PP^5$ be
a smooth cubic fourfold, then there is a semiorthogonal decomposition
$$\Db(X) = \langle \cat{A}_X, \ko_X, \ko_X(1), \ko_X(2) \rangle.$$
The category $\cat{A}_X$ is 2-Calabi--Yau.
\begin{conjecture}\label{kuz-conj}\bf(Kuznetsov). \rm
The cubic fourfold $X$ is rational if and only if $\cat{A}_X \simeq \Db(Y)$ for a smooth projective K3 surface $Y$.
\end{conjecture}
This conjecture has been verified in \cite{kuznetcubicfourfold} for singular cubics, pfaffian cubics and
Hassett's \cite{hassett-special} examples.
When $X$ contains a plane $P$
there is a way more explicit construction: blowing up $P$ we obtain a quadric bundle $\widetilde{X} \to \PP^2$ of
relative dimension 2, degenerating along a sextic. If the sextic is smooth, let $S \to \PP^2$ be the double
cover, which is a K3 surface. Then
$$\cat{A}_X \simeq \Db(\PP^2,\kb_0) \simeq \Db(S,\ka),$$
where $\kb_0$ is associated to the quadric fibration and $\ka$ is an Azumaya algebra, obtained lifting $\kb_0$
to $S$. The questions about categorical
representability of noncommutative varieties arise then very naturally in this context.
Notice that if $\cat{A}_X$ is representable in dimension 2, then we know something weaker than Kuznetsov conjecture:
we would have a smooth projective surface $S'$ and a fully faithful functor $\cat{A}_X \to \Db(S')$. Point
3) of Question \ref{question-cycat} appears naturally in this context.

\begin{question}\label{quest-kuz-conj}
One can then wonder if the Kuzentsov conjecture may be stated in the following form: the cubic fourfold $X$ is rational if and only if
it is categorically representable in codimension 2. This is equivalent to proving that the 2-Calabi--Yau category $\cat{A}_X$
is representable in dimension 2 if and only if there exist a K3 surface $Y$ and an equivalence $\Db(Y) \simeq \cat{A}_X$. 
\end{question}

We can propose some more examples of fourfolds for which a Kuznetsov-type conjecture seems natural: if $X$ is the complete intersection
of three quadrics $Q_1$, $Q_2$, $Q_3$ in $\PP^7$, then Homological Projective Duality (\cite{kuznetHPD,kuznetconicbundles}) gives an exceptional
sequence on $X$ and its complement $\cat{A}_X \simeq \Db(\PP^2,\kb_0)$, where $\kb_0$ is the even Clifford algebra associated to the net
of quadrics generated by $Q_1,Q_2,Q_3$. Similarly, if we consider two quadric fibrations $Q_1, Q_2 \to \PP^1$ of relative dimension
$4$ and their complete intersection $X$, there is an exceptional sequence
on $X$, and let $\cat{A}_X$ be its orthogonal complement. A realtive version of Homological Projective Duality shows that $\cat{A}_X$
equivalent to $\Db(S,\kb_0)$, where $S$ is a $\PP^1$-bundle over $\PP^1$ and $\kb_0$ the
even Clifford algebra associated to the net of quadrics generated by $Q_1$ and $Q_2$. It is natural to wonder if representability in dimension
2 of the noncommutative varieties is equivalent or is a necessary condition for rationality of $X$. A partial answer to the last
example will be provided in a forthcoming paper \cite{asher-marcello-bolo}.

Other examples in dimension 7 are provided in \cite{iliev-manivel-cubic-7-8folds}.
If $X$ is a cubic sevenfold, there is a distinguished subcategory $\cat{A}_X$ of $\Db(X)$,
namely the orthogonal complement of the exceptional sequence $\{\ko_X, \ldots, \ko_X(5)\}$. This
is a 3-Calabi--Yau category. If $X$ is Pfaffian, it is shown in \cite{iliev-manivel-cubic-7-8folds}
that $\cat{A}_X$ cannot be equivalent to the derived category of any smooth projective variety.
It is also conjectured that $\cat{A}_X$ is equivalent to the orthogonal complement of an exceptional
sequence in the derived category $\Db(Y)$ of a Fano sevenfold $Y$ of index 3,
birationally equivalent to $X$. 

\subsection{Other approaches}\label{approaches}

Of course categorical representability is just one among different approaches to the study
of birational geometry of a variety via derived categories. Nevertheless there is some
common ground.

\smallskip

First of all, Kuznetsov mentions in \cite{kuznetcubicfourfold} the notion of Clemens--Griffiths component of $\Db(X)$, whose
vanishing would be a necessary condition for rationality. It seems reasonable to expect that categorical representability
in codimension 2 implies the vanishing of the Clemens--Griffiths component.

Another recent theory is based on
Orlov spectra and their gaps \cite{katza-favero-ballard-generation-time}. Let us refrain even to sketch a definition of it,
but just notice that \cite[Conj. 2]{katza-favero-ballard-generation-time} draws a link between categorical
representability and gaps in the Orlov spectrum (see, in particular \cite[Cor. 1.11]{katza-favero-ballard-generation-time}).
Finally, conjectures based on homological mirror symmetry are proposed in
\cite{katza-rat-hms-1,katza-rat-hms-2}, but we cannot state a precise relation with our construction. A careful study
of the example constructed in \cite{katza-iliev-pry-nonrational} would be a good starting point.

\bibliographystyle{amsalpha}\bibliography{references}

\end{document}